\documentclass[11pt, reqno]{amsart}
\usepackage{amsfonts,latexsym,amsthm,amssymb,graphicx}
\usepackage[all]{xy}
\DeclareFontFamily{OT1}{rsfs}{}
\DeclareFontShape{OT1}{rsfs}{n}{it}{<-> rsfs10}{}
\DeclareMathAlphabet{\mathscr}{OT1}{rsfs}{n}{it}

\newtheorem{theorem}{Theorem}[section]
\newtheorem{lemma}[theorem]{Lemma}
\newtheorem{corol}[theorem]{Corollary}

{\theoremstyle{definition} }
{\theoremstyle{remark} \newtheorem{remark}[theorem]{Remark}

\newcommand{\Cbb}{{\mathbb{C}}}

\title{Chern Classes of Logarithmic Vector Fields}
\author{Xia Liao}
\address{
Mathematics Department, 
Florida State University,
Tallahassee FL 32306, U.S.A.
}
\email{xliao@math.fsu.edu}

\begin{document}
\maketitle

\begin{abstract}
Let $X$ be a nonsingular complex variety and $D$ a reduced effective divisor in $X$. In this paper we study the conditions under which the formula $c_{SM}(1_U)=c(\textup{Der}_X(-\log D))\cap [X]$ is true. We prove that this formula is equivalent to a Riemann-Roch type of formula. As a corollary, we show that over a surface, the formula is true if and only if the Milnor number equals the Tjurina number at each singularity of $D$. We also show the Rimann-Roch type of formula is true if the Jacobian scheme of $D$ is nonsingular or a complete intersection.
\end{abstract}

\section{introduction}
Let $X$ be a nonsingular variety over $\Cbb$, and $U$ be the complement of a free divisor $D$ in $X$. In this paper, we study the conditions under which the formula
\begin{equation}\label{formula}
c_{SM}(1_U)=c(\textup{Der}_X(-\log D))\cap [X]
\end{equation}
is true. The left hand side of the formula is the Chern-Schwartz-MacPherson class of the open subvariety $U$, and the right hand side is the total Chern class of the sheaf of logarithmic vector fields along $D$. Throughout this paper we are working over the Chow homology theory $A_*$. The Chern classes of vector bundles and coherent sheaves are treated as operators on the Chow ring as described in \cite{MR732620}.

We show that this formula is equivalent to an analogue of a Riemann-Roch type of formula. As a corollary, we show that on a surface, formula~\eqref{formula} is true if and only if the Tjurina number and the Milnor number are the same for all singularities of $D$.

The question is motivated by two previous results of Aluffi \cite{aluffi} \cite{MR1697199}, which state that if $D$ is a normal crossing divisor, or a free hyperplane arrangement of a projective space, then formula ~\eqref{formula} is true. Because normal crossing divisors are free divisors (for this fact and a definition of free divisors, see \cite{MR586450}), It is a natural question to ask if formula~\eqref{formula} holds for any free divisor $D$. The result of this paper implies the freeness of the divisor is in general insufficient to guarantee the validity of  formula~\eqref{formula}.

Replacing $X$ by a complex manifold, we get a complex analytic version of formula~\eqref{formula}. In this case, it is well known that the local quasi homogeneity of the divisor at an isolated singularity is characterized by the equality of the Tjurina number and the Milnor number\cite{MR0294699}. So the corollary about formula~\eqref{formula} over surfaces can be restated in a slight different manner: formula~\eqref{formula} is true for surfaces if and only if the divisor $D$ is locally quasi homogeneous.

\section{technical preparations}

\subsection{Chern-Schwartz-MacPherson class of a hypersurface}
Continue with the notations introduced in the previous section. Recall there is a unique natural transformation from the functor of constructible functions to the Chow functor, associating the characteristic function of a nonsingular variety to the total Chern class of its tangent bundle. Then $c_{SM}(1_U)$ is the image of $1_X-1_D$ in the Chow group of X. For the purpose of the calculation in this paper, we don't need a detailed description of the MacPherson transformation. What we need is a formula of the CSM classes of singular hypersurfaces proved in ~\cite{MR1697199}.

\begin{lemma}[\cite{MR1697199}]
Let $D$ be an effective divisor in a nonsingular variety $X$, and let $J_D$ be its Jacobian scheme. Then
\begin{equation}\label{hyper}
c_{SM}(1_D)=c(\textup{Der}_X) \cap (s(D,X)+c(\mathscr{O}(D))^{-1} \cap (s(J_D, X)^{\vee} \otimes_X D))
\end{equation}
\end{lemma} 

\begin{remark}\label{ten}
The following is a brief clarification of several notations in the formula and some properties we will use in the next section. Proofs and more detailed explanations can be found in \cite{MR1316973}
\begin{enumerate}
\item The Jacobian scheme $J_D$, also called the singular subscheme of the divisor $D$, is the subscheme locally defined by the equation of $D$ and all its partial derivatives.
\item For a subscheme $Y$ of $X$, $s(Y,X)$ is the {\em{Segre class}} of $Y$ in $X$ in the sense of   
\cite{MR732620}, Chapter 4. 
\item The Segre class $s(D,X)$ is easy to compute for an effective divisor $D$ in $X$. It equals the class $\frac{[D]}{c(\mathscr{O}(D))}$.   
\item If $A= \oplus_{i}a^i$ is a rational equivalence class on a scheme $X$, indexed by codimension, we let
\begin{equation*}
A^{\vee}=\sum_{i\geq 0} (-1)^i a^i
\end{equation*} 
the {\em{dual}} of $A$; and for a divisor $D$ we let
\begin{equation*}
A\otimes_X D=\sum_{i\geq 0} \frac{a^i}{c(\mathscr{O}(D))^i}
\end{equation*} 
the {\em{tensor}} of $A$ by $D$. The subscript of the tensor tells the ambient variety where the codimension of $A$ is calculated.
\item This tensor product of a rational equivalence class and a divisor satisfies the ``associative law''.
\begin{equation*}
(A\otimes_X D_1)\otimes_X D_2= A\otimes_X (D_1\otimes D_2)
\end{equation*}
\item if $\mathscr{F}$ is a coherent sheaf of rank 0 on a nonsingular variety $X$, then
\begin{equation*}
c(\mathscr{F} \otimes D)\cap [X]=(c(\mathscr{F}) \cap[X]) \otimes_X D
\end{equation*}
\end{enumerate}
\end{remark}

\subsection{sheaves of logarithmic vector fields}
Let $X$ be a nonsingular variety over a field $k$, and $D$ a reduced effective divisor on $X$. The sheaf of logarithmic vector fields along the divisor $D$ is a subsheaf of the sheaf of regular derivations. Over an open subset where the divisor $D$ has a local equation $f$, it is given by $\textup{Der}_X(-\log D)(U)= \{ \theta \in \textup{Der}_X(U) \,|\, \theta f \in (f) \}$. Saito was the first to study this sheaf in full generality\cite{MR586450}. He showed that $\textup{Der}_X(-\log D)$ is a reflexive sheaf. Its dual sheaf is called the sheaf of logarithmic 1-forms. Saito also gave a criterion when $\textup{Der}_X(-\log D)$ is locally free. A divisor is free if $\textup{Der}_X(-\log D)$ is locally free. Normal crossing divisors (in particular nonsingular divisors) are always free.

From the definition of the $\textup{Der}_X(-\log D)$, we obtain an exact sequence: 
\begin{equation*}
0 \to \textup{Der}_X(-\log D) \to \textup{Der}_X \to \mathscr{J}_D(D) \to 0
\end{equation*}
where $\mathscr{J}_D$ is called the Jacobian ideal of $D$. It is an ideal sheaf of $\mathscr{O}_D$
locally generated by $\theta f$ for all $\theta \in \textup{Der}_X(U)$. From this description, we see the second arrow in the above exact sequence takes each $\theta \in \textup{Der}_X(U)$ to $\theta f$ over an open subset $U$.

Combining this sequence with the sequence defining the Jacobian scheme $J_D$:
\begin{equation*}
0 \to \mathscr{J}_D \to \mathscr{O}_D \to \mathscr{O}_{J_D} \to 0
\end{equation*}
we get a long exact sequence
\begin{equation}\label{fun}
0 \to \textup{Der}_X(-\log D) \to \textup{Der}_X \to \mathscr{O}_D(D) \to \mathscr{O}_{J_D}(D) \to 0
\end{equation} 
 
Exact sequence \eqref{fun} implies that a singular divisor $D$ is free if and only if its Jacobian scheme $J_D$ is Cohen-Macaulay of codimension 2 \cite{MR981525} \cite{MR2184799}. In fact this can be seen from the following argument. To say $J_D$ is Cohen-Macaulay of codimension 2 is equivalent to say the $\mathscr{O}_{J_D}$ is a Cohen-Macaulay $\mathscr{O}_X$-module of dimension $n-2$, where $n=dimX$. This condition is true if and only if the depths of all stalks of $\mathscr{O}_{J_D}$ are $n-2$. The Auslander-Buschbaum formula tells us the projective dimension of $\mathscr{O}_{J_D}$ at each stalk is thus 2, which in turn is equivalent to the syzygy sheaf $\textup{Der}_X(-\log D)$ being locally free. This simple observation allows us to deduce the fact that any reduced effective divisor in a surface is free. In fact, the Jacobian scheme $J_D$ is 0-dimensional in this case. A 0-dimensional module is always Cohen-Macaulay because the depth is smaller or equal to the dimension for a finitely generated module over a local ring.

\section{A Riemann-Roch type formula}

\begin{theorem}\label{RRtype}
Formula ~\eqref{formula} is true if and only if a Riemann-Roch type formula $\pi_*(c(\mathscr{O}_E) \cap [\tilde X])=c(\mathscr{O}_{J_D}) \cap [X]$ holds. Here $J_D$ is the Jacobian scheme of $D$, $\tilde X$ is the blowup of X along $J_D$, $E$ is the exceptional divisor, and $\pi$ is the morphism $\tilde X \to X$. The same result can also be stated in terms of a comparison between a Segre class and a Chern class: Formula ~\eqref{formula} is true if and only if $[X]-s(J_D,X)^{\vee}=c(\mathscr{O}_{J_D}) \cap [X]$. 
\end{theorem}

\begin{proof}
The proof of the theorem is based on equation \eqref{hyper} and exact sequence \eqref{fun} in the previous section. Taking the total Chern class for exact sequence \eqref{fun}, we have:
\begin{equation*}
c(\textup{Der}_X(-\log D))=\frac{c(\textup{Der}_X)\cdot c(\mathscr{O}_{J_D}(D))}{c(\mathscr{O}_D(D))}=\frac{c(\textup{Der}_X)\cdot c(\mathscr{O}_{J_D}(D))}{c(\mathscr{O}(D))}
\end{equation*}
In the second equality we use the fact that $c(\mathscr{O}_D(D))=c(\mathscr{O}(D))$. This can be seen by tensoring $\mathscr{O}(D)$ to the exact sequence $0 \to \mathscr{O}(-D) \to \mathscr{O}_X \to \mathscr{O}_D \to 0$ and then taking Chern classes. Then
\begin{equation*}
c_{SM}(1_U)=c(\textup{Der}_X(-\log D))\cap [X] \iff
\end{equation*}

\begin{equation*} 
c(\textup{Der}_X) \cap \left([X]-\frac{[D]}{c(\mathscr{O}(D))}-\frac{s(J_D, X)^{\vee} \otimes_X D}{c(\mathscr{O}(D))}\right)=\frac{c(\textup{Der}_X)\cdot c(\mathscr{O}_{J_D}(D))}{c(\mathscr{O}(D))} \cap[X] \iff
\end{equation*}

\begin{equation*}
\frac{c(\textup{Der}_X)}{c(\mathscr{O}(D))} \cap \left([X]-s(J_D, X)^{\vee} \otimes_X D\right)=\frac{c(\textup{Der}_X)\cdot c(\mathscr{O}_{J_D}(D))}{c(\mathscr{O}(D))} \cap[X] \iff
\end{equation*}

\begin{equation*}
[X]-s(J_D, X)^{\vee} \otimes_X D=c(\mathscr{O}_{J_D}(D)) \cap [X] \iff
\end{equation*}

\begin{equation*}
[X]-s(J_D, X)^{\vee}= c(\mathscr{O}_{J_D}) \cap[X]
\end{equation*} 
In the last step, we ``tensor'' the classes on each side of the equation by the divisor $-D$. The tensor product of $[X]$ and $-D$ is $[X]$ itself because the codimension of the class $[X]$ is $0$. The tensor product of $S(J_D, X)^{\vee} \otimes_X D$ and $-D$ is $S(J_D, X)^{\vee}$ according to property (5) of \ref{ten}. The right hand side term of the equation is taken care of by property (6) of \ref{ten}, because $\mathscr{O}_{J_D}$ is a rank 0 $\mathscr{O}_X$-module. 

Next we want to prove the Riemann-Roch type formula. Recall the Segre class is preserved by proper morphisms of schemes: $\pi_*(s(E,\tilde X))=s(J_D,X)$\cite{MR732620}. It is also easy to see the dual notation of classes is compatible with the push forward of classes. Thus we have:
\begin{equation*}
\begin{split}
[X]-S(J_D, X)^{\vee}& =\pi_*([\tilde X]-s(E,\tilde X)^\vee)\\
                                    & =\pi_*\left([\tilde X]-(\frac{[E]}{1+E})^\vee\right)\\
                                    & =\pi_*\left([\tilde X]-\frac{-[E]}{1-E}\right)\\
                                    & =\pi_*\left((1+\frac{E}{1-E})\cap [\tilde X]\right)\\
                                    & =\pi_*\left(\frac{1}{1-E} \cap [\tilde X]\right)\\
                                    &= \pi_*(c(\mathscr{O}_E)\cap [\tilde X])
\end{split}
\end{equation*} 
In this computation, the notations is chosen so that $[E] \in A_{n-1}(\tilde X)$ and $E$ as a divisor is an abbreviation of $c_1(\mathscr{O}(E))$.
\end{proof}

\begin{corol}\label{surface}
Let $D$ be a reduced effective divisor on a nonsingular complex surface $X$. Then formula ~\eqref{formula} is true if and only if the Tjurina number equals the Milnor number for all singularities of $D$.    
\end{corol}

\begin{proof}
We compare the degree zero components of $c(\mathscr{O}_P) \cap [X]$ and $s(J_D,X)$. The corollary is based on the following results:
\begin{enumerate}
\item $s(J_D,X)=\sum \mu (P) [P]$ where the sum is taken over all singular points $P$ of the divisor $D$ and $\mu (P)$ is the milnor number of $P$.
\item $c(\mathscr{O}_{J_D})=\prod c(\mathscr{O}_P)^{l(P)}$ where the product is taken over all singular points $P$ of the divisor $D$. $l(P)$ is the length of the stalk of $\mathscr{O}_{J_D}$ at $P$ (because we have isolated singularities, each such stalk is an Artinian ring). These numbers are also called Tjurina numbers in literatures. $\mathscr{O}_P$ is the structure sheaf of the nonsingular subscheme supported at $P$.
\item $c(\mathscr{O}_P) \cap [X]=[X]-[P]$
\end{enumerate} 
The proof of property (1) can be found in section 7.1 of \cite{MR732620} and example 10.14 from \cite{MR2113135}. The ideas in the proof is that $s(J_D^{'}, X)$ is closely related to the Milnor class, and is equal to this class when $s(J_D^{'}, X)$ is a 0-dimensional cycle class. Here $J_D^{'}$ is the subscheme having the same support as $J_D$ but defined by a (in priori) smaller ideal generated by all partial derivatives of a local equation of $D$ (the ideal $J_D$ contains all partial derivatives of a local equation of $D$ as well as a local equation of $D$). On the other hand, $s(J_D^{'}, X)=s(J_D, X)$ because the ideal $J_D$ is integral over the smaller ideal $J_D^{'}$. As a result, $s(J_D, X)$ computes the Milnor numbers of the singular points. 

Property (2) is obtained by considering the sequences
\begin{equation*}
0 \to \mathscr{I}^j\slash \mathscr{I}^{j+1} \to \mathscr{I}^i \to \mathscr{I}^{i+1} \to 0
\end{equation*}
where $\mathscr{I}$ is the ideal sheaf of the Jacobian subscheme $J_D$ supported at a singular point $P$, giving $P$ the nonsingular scheme structure. In another word, $\mathscr{I}$ corresponds to the maximal ideal of the local ring of $\mathscr{O}_{J_D}$ at $P$. Because all stalks of $\mathscr{O}_{J_D}$ are Artinian rings, some big powers of $\mathscr{I}$ become 0, so we only have finite sequences to consider. Also notice the sheaves $\mathscr{I}^j\slash \mathscr{I}^{j+1}$ are free $\mathscr{O}_P$-modules of rank ${l(\mathscr{I}^j\slash \mathscr{I}^{j+1})}$. Taking Chern classes for these sequences, we get
\begin{equation*}
c(\mathscr{I}^j) \slash c(\mathscr{I}^{j+1})= c(\mathscr{I}^j\slash \mathscr{I}^{j+1})=c(\mathscr{O}_P)^{l(\mathscr{I}^j\slash \mathscr{I}^{j+1})}
\end{equation*}
Multiplying all such equations together, we get the desired result.

Another way to understand this result by K-thoery of coherent sheaves can be found in example 15.1.5 and 15.3.6 of \cite{MR732620}. 

Property (3) is a rather trivial case of the Riemann-Roch without denominators. A general discussion can be found in section 15.3 of \cite{MR732620}.
\end{proof}
As mentioning in the introduction, the Milnor number being equal to Tjurina number characterizes local quasi homogeneity for isolated singularities \cite{MR0294699}. The previous corollary can be stated in the following manner:

\begin{corol}
Let $D$ be a reduced effective divisor on a nonsingular complex surface $X$. Then formula ~\eqref{formula} is true if and only if the divisor $D$ is locally quasi homogeneous.
\end{corol} 
 
 \section{A further discussion of the Riemann-Roch type formula}
In previous section, we showed the original formula \eqref{formula} is true if and only if a formula concerning a Segre class is true, and it is moreover equivalent to a Riemann-Roch type formula. Although it is a byproduct of the original study, the last formula can be studied independent of the context of Chern classes of logarithmic vector fields. We ask the question: For what type of the subscheme $Y$ of a nonsingular scheme $X$ is the formula: 
\begin{equation*}
[X]-s(Y, X)^{\vee}=c(\mathscr{O}_{Y}) \cap[X]
\end{equation*}
true? 

The formula can be easily tested true if $Y=D$ is a divisor. In fact, both sides equal to $\frac{[X]}{c(\mathscr{O}(-D))}$ in this case. the Segre class $s(Y, X)=\frac{[D]}{c(\mathscr{O}(D))}$ according to item (3) of \ref{ten} so the dual $s(Y, X)^{\vee}=\frac{-[D]}{c(\mathscr{O}(-D))}$. On the other hand $c(\mathscr{O}_D)=\frac{[X]}{c(\mathscr{O}(-D))}$ by taking Chern classes on the exact sequence $0 \to \mathscr{O}(-D) \to \mathscr{O}_X \to \mathscr{O}_D \to 0$. 

Moreover, we have a little deeper result:

\begin{theorem}
The above formula is true if $Y$ is regularly embedded in $X$ of codimension 2.
\end{theorem}

\begin{proof}
Notice in these cases, the subscheme $Y$ has the normal bundle in $X$. Let $N$ be the normal bundle. We have an easy expression of the Segre class in terms of the normal bundle:
\begin{equation*}
s(Y,X)=c(N)^{-1} \cap [Y]
\end{equation*}

Then
\begin{equation*}
[X]-s(Y, X)^{\vee}=[X]-c(N^{\vee})^{-1} \cap (-1)^2[Y]=[X]-c(N^{\vee})^{-1} \cap [Y]
\end{equation*}

By an application of Riemann-Roch without denominators, we can also show
\begin{equation*}
c(\mathscr{O}_Y) \cap[X]=[X]-c(N^{\vee})^{-1} \cap [Y].
\end{equation*}
More details can be found in example 15.3.5 and example 18.2.1 from \cite{MR732620}.
\end{proof}

\begin{remark}
By the same reference \cite{MR732620} example 15.3.5, we can also show that this theorem is not true for regular embedding of codimension 3 or higher. For example in codimension 3, we have:
\begin{equation*}
[X]-s(Y, X)^{\vee}=[X]-c(N^{\vee})^{-1} \cap (-1)^3[Y]=[X]+c(N^{\vee})^{-1} \cap [Y]
\end{equation*}
and
\begin{equation*}
c(\mathscr{O}_Y) \cap[X]=[X]+c(N^{\vee})^{-1}(2-c_1(N))(1-c_1(N))^{-1}\cap [Y]
\end{equation*}
so apparently these two expressions are not the same. If the embedding is of higher codimension, the Riemann-Roch formula for computing $c(\mathscr{O}_Y)\cap[X]$ is even more complicated. 
\end{remark}

\begin{corol}
Formula \eqref{formula} is true if the Jacobian scheme of the divisor $D$ is regularly embedded in $X$ of codimension 2. (The freeness of the divisor $D$ is automatic by the conditions in this case.)  
\end{corol}

\bibliographystyle{alpha}
\bibliography{liaobib}
  
\end{document}